\documentclass[a4paper,12pt]{amsart}
\usepackage[utf8x]{inputenc}
\usepackage{mathrsfs}
\usepackage{amsfonts,amssymb,amscd,amsthm,amsmath,graphicx}
\usepackage{longtable}
\usepackage[all]{xy} 
\usepackage{verbatim}

\newtheorem{theorem}{Theorem}[section]

\newtheorem{lemma}[theorem]{Lemma}

\newtheorem{prop}[theorem]{Proposition}

\newtheorem{definition}[theorem]{Definition}
\newtheorem{remark}[theorem]{Remark}
\newtheorem{example}[theorem]{Example}

\def\E{\operatorname{E}}
\def\F{\operatorname{F}}
\def\G{\operatorname{G}}
\def\U{\operatorname{U}}
\def\O{\operatorname{O}}
\def\SU{\operatorname{SU}}
\def\SO{\operatorname{SO}}
\def\Sp{\operatorname{Sp}}
\def\PO{\operatorname{PO}}

\def\PSO{\operatorname{PSO}}
\def\PSU{\operatorname{PSU}}

\def\Ad{\operatorname{Ad}}
\def\ad{\operatorname{ad}}

\def\Aut{\operatorname{Aut}}

\def\der{\operatorname{der}}

\def\id{\operatorname{id}}

\def\Int{\operatorname{Int}}

\def\PSp{\operatorname{PSp}}
\def\rank{\operatorname{rank}}

\def\Spin{\operatorname{Spin}}
\def\Stab{\operatorname{Stab}}

\newcommand{\fre}{\mathfrak{e}}
\newcommand{\frf}{\mathfrak{f}}
\newcommand{\frg}{\mathfrak{g}}

\newcommand{\frk}{\mathfrak{k}}

\newcommand{\fru}{\mathfrak{u}}

\newcommand{\bbR}{\mathbb{R}}

\begin{document}

\title{Maximal antipodal sets in irreducible compact symmetric spaces}
%\date{August 2020}
\thanks{}

\author{Jun Yu}
\address{BICMR, Peking University, No. 5 Yiheyuan Road, Haidian District, Beijing 100871, China.}
\email{junyu@bicmr.pku.edu.cn}

\begin{abstract}
We give an explicit classification of maximal antipodal sets in any irreducible compact symmetric space except for
spin groups and half spin groups, and some quotient symmetric spaces associated to them.
\end{abstract}

\maketitle

\setcounter{tocdepth}{1}

\tableofcontents

%%%%%%%%%%%%%%%%%%%%%%%%%%%%%%%%%%%%%%%%%%%%%%%%%%%%%%%%%%%%%%%%%%%%%%%%%%%%%%
\section{Introduction}
%%%%%%%%%%%%%%%%%%%%%%%%%%%%%%%%%%%%%%%%%%%%%%%%%%%%%%%%%%%%%%%%%%%%%%%%%%%%%%%%%

A closed Riemannian manifold $M$ is said to be a compact symmetric space if for any point $x\in M$, there is
a Riemannian isometry $s_{x}: M\rightarrow M$ such that: (i) $s_{x}=x$; (ii) the tangent map $(s_{x})_{\ast}:
T_{x}(M)\rightarrow T_{x}(M)$ is $-1$. For any compact symmetric space $M$, it is known that there exists a
connected compact Lie group $G$ and an involutive automorphism $\theta$ of it such that $M=G/G^{\theta}$
(cf. \cite[Thm. 4.6, p. 185]{Loos1}). We call a nonempty subset $X$ of $M$ an {\it antipodal set} if
\[s_{x}(y)=y\ (\forall x,y\in X).\] An antipodal set must be a finite set since it is a discrete set and
$M$ is compact. We call an antipodal set a {\it maximal antipodal set} if it is not properly contained in
any other antipodal set. In \cite{Chen-Nagano}, Chen and Nagano introduced and calculated the invariant
2-number $\#_{2}(M)$ of a compact symmetric space, which is the maximal
cardinality of antipodal sets in a compact symmetric space $M$. After this paper, there are many studies on maximal
antipodal sets. Particularly, Tanaka and Tasaki made the classification of maximal antipodal sets for some kinds
of compact symmetric spaces (\cite{Tanaka}, \cite{Tanaka-Tasaki1}, \cite{Tanaka-Tasaki2}, \cite{Tanaka-Tasaki3}):
symmetric $R$-spaces, some compact classical Lie groups, etc. The readers may consult \cite{Chen} for an excellent
survey on the study of 2-numbers and maximal antipodal sets. In this paper we deduce the classification of maximal
antipodal sets from the classification of elementary abelian 2-subgroups in compact Lie groups (\cite{Yu}).

Let $G$ be a connected compact Lie group, and $\theta$ be an involutive automorphism of it. Put $M=G/G^{\theta}$.
Set $\bar{G}=G\rtimes\langle\bar{\theta}\rangle,$ where $\bar{\theta}^{2}=1$ and $\Ad(\bar{\theta})|_{G}=\theta.$
Write $C_{\bar{\theta}}=\{g\bar{\theta}g^{-1}: g\in G\}.$ The {\it Cartan quadratic morphism} (cf. \cite{Chen})
is a map $\phi: G/G^{\theta}\rightarrow G$ defined by \[\phi(gG^{\theta})=g\theta(g)^{-1},\ \forall g\in G.\]
Let $X$ be a subset of $M$ containing the origin $o=eG^{\theta}\in M$. Write \[\phi(X)=\{\phi(x): x\in X\}
\subset G,\] \[F_{1}(X)=\langle\phi(X)\rangle\subset G\] and \[F_{2}(X)=\langle\phi(X),\bar{\theta}\rangle
\subset\bar{G}.\] Using the Cartan quadratic morphism, we show a correspondence between maximal antipodal sets
in $G/G^{\theta}$ and certain elementary abelian 2-subgroups of $\bar{G}$.

\begin{theorem}\label{T1}
Let $X$ be a subset of $M=G/G^{\theta}$ containing the origin $o=eG^{\theta}\in M$. Then $X$ is a maximal antipodal
set if and only if $F_{2}(X)$ is a maximal element in the set of elementary abelian $2$-subgroups of $\bar{G}$
which are generated by elements in $C_{\bar{\theta}}$, and \begin{equation*}X=\{x\in M: \phi(x)\in F_{2}(X)
\cap C_{\bar{\theta}}\bar{\theta}^{-1}\}.\end{equation*}
\end{theorem}

We call a compact symmetric space ``irreducible'' if it is not isogenous to the product of two positive-dimensional
compact symmetric spaces. With an explicit list of irreducible compact symmetric spaces, we show a precise
classification of maximal antipodal sets in most of them using Theorem \ref{T1}. The only cases which haven't been
treated are spin groups and half spin groups, and some quotient symmetric spaces of them.

%\smallskip

The content of this paper is organized as follows. In Proposition \ref{P:characterization2}, we give a criterion of
antipodal sets using the Cartan quadratic morphism $\phi: G/G^{\theta}\rightarrow G$. With it, we show Theorem
\ref{T1}. In Subsection \ref{SS:Weyl}, we study Weyl groups of maximal antipodal sets. In Section
\ref{S:symmetric space}, we give a precise list of irreducible compact symmetric spaces that are not of group form.
In Section \ref{S:classification}, we present an explicit classification of maximal antipodal sets in most irreducible
compact symmetric spaces. The remaining ones which haven't been treated are listed in Subsection \ref{SS:open}.
In Subsection \ref{S:anitipodal2}, we illustrate how to classify $G^{\theta}$ orbits in the fixed point set of $s_{0}$,
which are related to polars defined by Chen and Nagano.

\smallskip

\noindent{\it Notation and conventions.} In this paper a compact Lie group $G$ is said to be ``simple" if its Lie
algebra is a non-abelian simple Lie algebra. Let $\E_6^{sc}$ (or $\E_6$) denote a connected and simply-connected
compact simple Lie group of type $\E_6$; let $\E_{6}^{ad}$ denote a connected adjoint type compact simple Lie group
of type $\E_6$. Similarly, we have the notations $\E_7^{sc}$, $\E_7$, $\E_7^{ad}$, $\E_8$, $\F_4$, $\G_2$. The last
three are connected compact Lie groups which are both simply-connected and of adjoint type.

Write \[J_{m}=\left(\begin{array}{cc} 0&I_{m}\\-I_{m}&0\\\end{array}\right),\quad I_{p,q}=\left(\begin{array}{cc}-I_{p}
&0\\0&I_{q}\\\end{array}\right).\] In $\Spin(2n)$, write $c=c_{n}=e_{1}\cdots e_{2n}$, where $\{e_1,e_2,\dots,e_{2n}\}$
is a standard normal basis of the Euclidean space based on which $\Spin(2n)$ is defined. Involutive automorphisms
$\sigma_{i}$ of compact exceptional simple Lie algebras are as specified in \cite[Table 1]{Huang-Yu}.

Write $\omega_{m}=e^{\frac{2\pi i}{m}}$, which is a primitive $m$-th root of unity.

\smallskip

\noindent{\it Acknowledgements.} A part of this work was done when the author visited MPI Bonn in the summer of 2016 and
a draft was written when the author visited National University of Singapore in January 2018. The author would like to
thank both institutions for their support and hospitality. I would like to thank the referees for providing useful
references and giving very helpful comments. This research is partially supported by the NSFC Grant 11971036.

%%%%%%%%%%%%%%%%%%%%%%%%%%%%%%%%%%%%%%%%%%%%%%%%%%%%%%%%%%%%%%%%%%%%%%%%%%%%%%
\section{Characterization of antipodal sets}\label{S:method}
%%%%%%%%%%%%%%%%%%%%%%%%%%%%%%%%%%%%%%%%%%%%%%%%%%%%%%%%%%%%%%%%%%%%%%%%%%%%%%%%

\begin{comment}
\subsection{A criterion for antipodal sets}\label{SS:antipodal1}

Let $G$ be a connected compact Lie group, $\theta$ an involutive automorphism of $G$, and $H$ a closed subgroup of $G$ with
$(G^{\theta})^{0}\subset H\subset G^{\theta}.$ Write $M=G/H$, which is a compact symmetric space. Let $o=eH$ denote
the {\it origin} in $M=G/H$. For any $x=gH\in M=G/H$, set $\phi(x)=g\theta(g)^{-1}.$ Apparently, $\phi(x)$ does not depend
on the choice of $g$. Thus, we have a well-defined map $$\phi: M=G/H\rightarrow G,$$ which is called the {\it Cartan
quadratic morphism} in the literature (cf. \cite{Chen}).

\begin{prop}\label{P:characterization1}
A nonempty subset $X$ of $M=G/H$ is an antipodal set if and only if $\phi(g_{2}^{-1}g_{1})\in H$ for any points
$x_{1}=g_{1}H\in M=G/H$ and $x_{2}=g_{2}H\in M=G/H.$
\end{prop}

\begin{proof}
As $L_{g_1}(o)=g_{1}H=x_1$, we have $$s_{x_1}=L_{g_1}s_{o}L_{g_1}^{-1}.$$ Then, $$s_{x_1}(x_2)=L_{g_1}s_{o}L_{g_1}^{-1}
(g_2H)=L_{g_1}s_{o}(g_{1}^{-1}g_2H)=L_{g_1}(\theta(g_{1}^{-1}g_{2})H)=g_{1}\theta(g_{1}^{-1}g_{2})H.$$ Thus, $x_2=
s_{x_1}(x_2)$ if and only if $g_{1}\theta(g_{1}^{-1}g_{2})H=g_{2}H.$ That is equivalent to $$\phi(g_{2}^{-1}g_{1})=
g_{2}^{-1}g_{1}\theta(g_{1}^{-1}g_{2})\in H.$$ This shows the conclusion of the proposition.
\end{proof}

If a nonempty subset $X$ of $M$ is an antipodal set, then so is $g\cdot X$ for any $g\in G$. Without loss of generality we
may assume that $o=eH\in X$.
\end{comment}

\subsection{Proof of Theorem \ref{T1}}\label{SS:antipodal2}

Let $G$ be a connected compact Lie group and $\theta$ be an involutive automorphism of it. Write $H=G^{\theta}$. Put
$M=G/G^{\theta}$, which is a compact symmetric space. Let $o=eG^{\theta}$ denote the {\it origin}. There is a left $G$
action on $G/G^{\theta}$ through $$L_{g}(g'G^{\theta})=g\cdot g'G^{\theta}=gg'G^{\theta},$$ and there is a $G$-action
on itself through $$g\ast g'=gg'\theta(g)^{-1}.$$ The Cartan quadratic morphism map $\phi$ is $G$-equivariant with
regard to these two actions, i.e., $$\phi(g\cdot x)=g\ast\phi(x),\ \forall g\in G,\forall x\in G/H.$$ It is clear
that $\phi$ is an imbedding. Apparently, the translation by any element in $G$ of an antipodal set in $G/G^{\theta}$
is still an antipodal set.

%\footnotetext{Alternatively, we could choose an element $c\in(Z_{G})^{\theta}$ and set $$\bar{G}_{c}=G\rtimes\langle
%\bar{\theta}\rangle,$$ where $\bar{\theta}^{2}=c$ and $\Ad(\bar{\theta})|_{G}=\theta.$ In this way,
%$\bar{G}_{c}$ may be a group which is more familiar to us than $\bar{G}$. Then, in Theorem \ref{T:characterization3}
%elow, $F_{2}(X)$ is an abelian subgroup of $\bar{G}_{c}$ generated by elements in $C_{\bar{\theta}}$, not
%necessarily an elementary abelian 2-subgroup.}

\begin{prop}\label{P:characterization2}
Let $X$ be a subset of $M$ containing the origin $o=eH\in M=G/H$. Then $X$ is an antipodal set if and only if
$\phi(x)\in H$ and $\phi(x)^{2}=1$ for any $x\in X$, and $\phi(x)$ commutes with $\phi(y)$ for any $x,y\in X$.
\end{prop}

\begin{proof}
We first show that: $X$ is an antipodal set if and only if $\phi(g_{2}^{-1}g_{1})\in H$ for any two points $x_{1}=
g_{1}H\in G/H$ and $x_{2}=g_{2}H\in G/H$. Note that \[s_{o}(gH)=\theta(g)H,\ \forall g\in G.\] Since $L_{g_1}(o)=
g_{1}H=x_1$, we have $s_{x_1}=L_{g_1}s_{o}L_{g_1}^{-1}.$ Then, $$s_{x_1}(x_2)=L_{g_1}s_{o}L_{g_1}^{-1}(g_2H)=
L_{g_1}s_{o}(g_{1}^{-1}g_2H)=L_{g_1}(\theta(g_{1}^{-1}g_{2})H)=g_{1}\theta(g_{1}^{-1}g_{2})H.$$ Thus, $x_2=
s_{x_1}(x_2)$ if and only if $g_{1}\theta(g_{1}^{-1}g_{2})H=g_{2}H.$ This is equivalent to $$\phi(g_{2}^{-1}g_{1})
=g_{2}^{-1}g_{1}\theta(g_{1}^{-1}g_{2})\in H.$$

Necessarity. Suppose $X$ is an antipodal set. Write $x=gH\in X$. Taking $x_1=x$ and $x_2=o$, we get $\phi(x)=
g\theta(g)^{-1}\in H$. That is to say, $\theta(\phi(x))=\phi(x)$. We also have $$\theta(\phi(x))=
\theta(g\theta(g)^{-1})=\theta(g)g^{-1}=\phi(x)^{-1}.$$ Thus, $\phi(x)^{2}=1$. Taking $x=g_{1}H\in X$ and $y=g_{2}H
\in X$, we get $\phi(g_{2}^{-1}g_{1})\in H$. By the argument above this leads to $\phi(g_{2}^{-1}g_{1})^{2}=1.$
Equivalently, $$(g_{2}^{-1}g_{1}\theta(g_{1}^{-1})\theta(g_{2}))^{2}=1.$$ This is equivalent to
$(\phi(x)\phi(y)^{-1})^{2}=1$. Since $\phi(x)^2=\phi(y)^2=1$, it follows that: $\phi(x)$ commutes with $\phi(y)$.

Sufficiency. Suppose $\phi(x)\in H$ and $\phi(x)^{2}=1$ for any $x\in X$, and $\phi(x)$ commutes with $\phi(y)$ for
any $x,y\in X$. For any two points $x,y\in M$, write $x=g_{1}H$ and $y=g_{2}H$. Reverse to the above argument, by
the conditions of $\phi(x)^{2}=\phi(y)^{2}=1$ and $\phi(x)$ commutes with $\phi(y)$, one gets
$\phi(g_{2}^{-1}g_{1})^{2}=1$. Again by the above argument, this is equivalent to $\phi(g_{2}^{-1}g_{1})\in H$.
Then, $X$ is an antipodal set.
\end{proof}

\begin{proof}[Proof of Theorem \ref{T1}]
Assume that $X$ is a maximal antipodal set. By Proposition \ref{P:characterization2}, $F_{2}(X)$ is an elementary
abelian $2$-subgroup of $\bar{G}$ generated by elements in $C_{\bar{\theta}}$. Write $$X'=\{x\in M: \phi(x)\in
F_{2}(X)\cap C_{\bar{\theta}}\bar{\theta}^{-1}\}.$$ Then, $X\subset X'$ and $F_{2}(X')\subset F_{2}(X)$. By
Proposition \ref{P:characterization2}, $X'$ is an antipodal set. By the maximality of $X$, we get $X=X'$. By a
similar argument, one shows that $F_{2}(X)$ is a maximal element in the set of elementary abelian $2$-subgroups
of $\bar{G}$ which are generated by elements in $C_{\bar{\theta}}$. The converse is clear.
\end{proof}

Note that each elementary abelian 2-subgroup is contained in a maximal one. In practice, we first classify maximal
elementary abelian 2-subgroups of $\bar{G}$ containing $\bar{\theta}$ up to conjugacy. Take such an $F$ and let
$X$ be the set elements $x\in M$ such that $\phi(x)\in F$. Then, we remove such $X$ which are not maximal and
leave only the maximal ones. In this way, we get all maximal antipodal sets in $G/G^{\theta}$ up to conjugacy.

\subsection{Weyl group}\label{SS:Weyl}

Define a map $\psi: G/H\rightarrow\bar{G}$ by $$\psi(gH)=g\bar{\theta}g^{-1}.$$ Let $X$ be a subset of $M=G/H$,
not necessarily contain the origin. Put $\psi(X)=\{\psi(x): x\in X\}.$ By Proposition \ref{P:characterization2},
one can show that $X$ is an antipodal set in $M$ if and only if $\psi(X)$ generates an elementary abelian
2-subgroup of $G$. Let it be still denote by $F_{2}(X)$. Set \[N_{G}(X)=\{g\in G: g\cdot X=X\},\] \[Z_{G}(X)=
\{g\in G:g\cdot x=x,\ \forall x\in X\}\] and $W_{G}(X)=N_{G}(X)/Z_{G}(X)$. Apparently, the conjugation action
of any $g\in N_{G}(X)$ on $G$ stabilizes $F_{2}(X)$, and the inducing action on $F_{2}(X)$ is trivial if and
only if $g\in Z_{G}(X)$. Thus, we have an injective homomorphism $W_{G}(X)\hookrightarrow W_{G}(F_{2}(X))$.
It is clear that \[W_{G}(X)=\{w\in W_{G}(F_{2}(X)): w\cdot \psi(X)=\psi(X)\}.\]

\begin{prop}
If $X$ is a maximal antipodal set in $M$, then $W_{G}(X)=W_{G}(F_{2}(X))$.
\end{prop}

\begin{proof}
Write $X'=\{x\in M: \psi(x)\in F_{2}(X)\}$. Then, $X'\supset X$. Using Proposition \ref{P:characterization2} one
can show that $X'$ is an antipodal set. By the maximality of $X$, we get $X=X'$. Sine the conjugation action of
each $w\in W_{G}(F_{2}(X))$ on $F_{2}(X)$ preserves conjugacy classes, it stabilizes $C_{\bar{\theta}}
\cap F_{2}(X)=\psi(X')=\psi(X)$. Hence, $$W_{G}(X)=\{w\in W_{G}(F_{2}(X)): w\cdot \psi(X)=\psi(X)\}=
W_{G}(F_{2}(X)).$$
\end{proof}

\subsection{Irreducible compact symmetric spaces of adjoint type}\label{SS:antipodal3}

Now assume that $G$ is a connected compact simple Lie group of {\it adjoint type}. Let $\fru_0$ be the Lie algebra
of $G$, which is a compact simple Lie algebra. Then, $G\cong\Int(\fru_0)$. For simplicity we identify $G$ with
$\Int(\fru_0)$, and regard $\theta$ as an element of $\Aut(\fru_0)$ which acts on $G=\Int(\fru_0)$ by conjugation.
Divide the discussion into two cases: (i)$\theta$ is an inner automorphism; (ii)$\theta$ is an outer automorphism.
In the first case, $\theta\in\Int(\fru_0)=G$ and $\bar{\theta}\theta^{-1}$ is a central element of $\bar{G}$. Thus, $\bar{G}=G\times\langle\bar{\theta}\theta^{-1}\rangle.$ Let $\pi:\bar{G}\rightarrow\Int(\fru_0)=G$ be the adjoint
homomorphism. Then $\pi|_{G}=\id$ and $\ker\pi=\langle\bar{\theta}\theta^{-1}\rangle$. Write \[C_{\theta}=
\{g\theta g^{-1}: g\in\Int(\fru_0)\}\subset\Int(\fru_0).\] Let $F(X)=p(F_{2}(X))$. Then $\theta\in F(X)$ and
$F(X)$ is an elementary abelian 2-subgroup of $\Int(\fru_0)$ generated by elements in $C_{\theta}$.
In the second case, $\theta\in\Aut(\fru_0)-\Int(\fru_0)$. We could identify $\bar{\theta}$ with $\theta\in
\Aut(\fru_0)$ and regard $\bar{G}$ as a subgroup of $\Aut(\fru_{0}).$ Let $F(X)=F_{2}(X)$. Then $\theta\in F(X)$
and $F(X)$ is generated by elements in $$C_{\theta}=\{g\theta g^{-1}:g\in\Int(\fru_0)\}\subset\Aut(\fru_0).$$

The following theorem follows from Theorem \ref{T1} directly.

\begin{theorem}\label{T:X-F}
Let $X$ be a subset of $M=\Int(\fru_0)/\Int(\fru_0)^{\theta}$ containing the origin. Then, $X$ is a maximal
antipodal set if and only if $F(X)$ is a maximal element in the set of elementary abelian 2-subgroups of
$\Aut(\fru_0)$ generated by elements in $C_{\theta}$ and \[X=\{gH:g\theta g^{-1}\in F(X)\}.\] %\label{Eq:X-F}
%\begin{equation}\label{Eq:F2-F} F_{2}(X)=\langle\{g\bar{\theta}g^{-1}: g\in G,g\theta g^{-1}\in F(X)\}\rangle,
%\end{equation}
%\begin{equation}\label{Eq:F1-F}F_{1}(X)=\langle\{g\bar{\theta}g^{-1}\bar{\theta}^{-1}:g\in G,g\theta g^{-1}
%\in F(X)\}\rangle,\end{equation}

%Conversely, if $F(X)$ is a maximal element in the set of elementary abelian 2-subgroups of $\Aut(\fru_0)$ which
%contain $\theta$ and are generated by elements in $C_{\theta}$, then $X=\{gH:g\theta g^{-1}\in F(X)\}$ is a maximal
%antipodal set in $M$.
\end{theorem}

\begin{comment}
\begin{proof}
By Theorem \ref{T1}, $F(X)$ is an elementary abelian 2-subgroup of $\Aut(\fru_0)$ generated
by elements in $C_{\theta}$. Other parts of the first statement follow from Equation (\ref{Eq:F2-F}).
Precisely to say, by Theorem \ref{T1}, $F_{2}(X)$ is a maximal element in the set of
elementary abelian $2$-subgroups of $\bar{G}$ which are generated by elements in
$C_{\bar{\theta}}$. By Equation (\ref{Eq:F2-F}), $F_{2}(X)$ and $F(X)$ determine each other. Thus,
$F(X)$ is a maximal element in the set of elementary abelian 2-subgroups of $\Aut(\fru_0)$ which are
generated by elements in $C_{\theta}$. By the definitions of $F_{1}(X)$, $F_{2}(X)$, $F(X)$, and the
maximality of $X$, Equations (\ref{Eq:F1-F})-(\ref{Eq:X-F}) follow from Equation (\ref{Eq:F2-F}).

Now we show Equation (\ref{Eq:F2-F}). Write $$F'_2(X)=\langle\{g\bar{\theta}g^{-1}: g\in G, g
\theta g^{-1}\in F(X)\}\rangle.$$ By the definition of $F_{2}(X)$, $F(X)$ and $F'_{2}(X)$, we get
$F_{2}(X)\subset F'_{2}(X)$. Thus, $\bar{\theta}\in F'_{2}(X)$ as $\bar{\theta}\in F_{2}(X)$.
Then, $F'_{2}(X)$ is generated by elements in $C_{\bar{\theta}}$. By Theorem \ref{T1},
$F_{2}(X)$ is a maximal element in the set of elementary abelian $2$-subgroup of $\bar{G}$ which are
generated by elements in $C_{\bar{\theta}}$. Thus, $F_{2}(X)=F'_{2}(X)$.

The converse statement is clear.
\end{proof}
\end{comment}

\begin{remark}\label{R:X-F}
With Theorem \ref{T:X-F}, we can deduce the classification of maximal antipodal sets in
$\Int(\fru_0)/\Int(\fru_0)^{\theta}$ from the classification of elementary abelian 2-subgroups of $\Aut(\fru_0)$
given in \cite{Yu}. It is only a routine work, we omit the details here. Note that conjugacy classes of elements
of each elementary abelian 2-subgroup of $\Aut(\fru_0)$ are described well in \cite{Yu}.
\end{remark}

%%%%%%%%%%%%%%%%%%%%%%%%%%%%%%%%%%%%%%%%%%%%%%%%%%%%%%%%%%%%%%%%%%%%%%%%%%%%%%
\section{A precise list of irreducible compact symmetric spaces}\label{S:symmetric space}
%%%%%%%%%%%%%%%%%%%%%%%%%%%%%%%%%%%%%%%%%%%%%%%%%%%%%%%%%%%%%%%%%%%%%%%%%%%%%%%%

Analogous to Lie groups, we use coverings to define isogeny for symmetric spaces.

\begin{definition}\label{D:isogeny}
Two compact symmetric spaces $M_1,M_2$ are said to be isogenous if they admit isomorphic universal coverings.
\end{definition}

We define irreducible symmetric spaces as follows.

\begin{definition}\label{D:irreducible}
A compact symmetric space $M$ is said to be irreducible if there exists no positive-dimensional compact symmetric
spaces $M_1,M_2$ such that $M$ is isogenous to $M_{1}\times M_{2}$.
\end{definition}

The following theorem is from \cite{Loos1}, which is pointed out to the author by an anonymous referee.

\begin{theorem}[\cite{Loos1}, p. 145, Theorem 4.6]\label{T:standard-form}
Let $M$ be a compact symmetric space. Then there is a compact Lie group $G$ and an involutive automorphism
$\theta$ of $G$ such that $M\cong G/G^{\theta}$.
\end{theorem}

One can show that (for example, use Theorem \ref{T:standard-form}) any irreducible compact symmetric space $M$ is
isomorphic to one of the following: (i)$S^{1}$; (ii)a compact simple Lie group; (iii)$G/G^{\theta}$ with $G$ a compact
simple Lie group and $\theta$ an involutive automorphism of it. Compact symmetric spaces in cases (i)-(ii) are said
to be of {\it group form}.

\begin{definition}\label{D:semisimple}
Let $M$ be a compact symmetric space. We call $M$ semisimple if its fundamental group $\pi_{1}(M)$ is finite; we call
$M$ simply-connected if $\pi_{1}(M)=1$; we call $M$ of adjoint type if there is no proper Riemannian covering
$M\rightarrow M'$ for $M'$ another compact symmetric space.
\end{definition}

In this section we give an explicit list of irreducible compact symmetric spaces that are not of group form by using
Theorem \ref{T:standard-form} and calculating symmetric subgroups $G^{\theta}$ (cf. \cite[Table 2, p. 408]{Huang-Yu}).
Recall that simply-connected compact symmetric spaces are classified by \'Elie Cartan and can be found in the
classical textbooks like \cite{Helgason}, \cite{Loos2}, \cite{Wolf}. The description and construction of non-simply
connected compact symmetric spaces are given in some excellent monographs (cf. \cite[Thm. 4.5, p. 103]{Borel},
\cite[Thm. 9.1, p. 326]{Helgason}, \cite[Proposition 2.4, p. 68-69]{Loos2}, \cite[Thm. 8.3.11, p. 244]{Wolf}).

\smallskip

\noindent\textbf{1, Grassmannians.} Put $c=e_{1}\dots e_{n}\in\Spin(n)$ and $L_{2n}=\frac{1+e_{1}e_{2n+1}}{\sqrt{2}}
\cdots\frac{1+e_{2n}e_{4n}}{\sqrt{2}}\in\Spin(4n)$. Any irreducible compact symmetric space $M$ which is isogenous to
a (real, complex or quaternion) Grassmannian is isomorphic to $G/G^{\theta}$ for some $(G,\theta)$ as in the following
list: \begin{enumerate}
\item[(i)] adjoint type: $G=\PSU(p+q)$, $\PSO(p+q)$ or $\PSp(p+q)$ ($q\geq p\geq 1$), $\theta=\Ad(I_{p,q})$.
\item[(ii)] $G=\SU(2p)$ ($p\geq 1$), $\theta=\Ad(I_{p,p})$, and $G^{\theta}=S(U(p)\times\U(p))$.
\item[(iii)] $G=\Sp(2p)$ ($p\geq 1$), $\theta=\Ad(I_{p,p})$, and $G^{\theta}\cong\Sp(p)\times\Sp(p)$.
\item[(iv)] $G=\SO(2p)$ ($p\geq 4$), $\theta=\Ad(I_{p,p})$, and $G^{\theta}=S(\O(p)\times\O(p))$.
\item[(v)] $G=\Spin(p+q)$ ($q\geq p\geq 1$), $\theta=\Ad(e_{1}\dots e_{p})$, $G^{\theta}=\Spin(p)\cdot \Spin(q))$.
\item[(vi)] $G=\Spin(4n)/\langle c\rangle$ ($n\geq 2$), $\theta=\Ad(e_{1}\dots e_{2n})$, and $G^{\theta}=((\Spin(2n)
\!\cdot\!\Spin(2n))\!\rtimes\!\langle\!L_{2n}\!\rangle)/\langle c\rangle$.
\end{enumerate}

\smallskip

\noindent\textbf{2, Types AI and AII.} Write $G_{n,m}=\SU(n)/\langle\omega_{m}I\rangle$ for any integer $m|n$. Put $J_{k}=\left(\begin{array}{cc}0_{k}&I_{k}\\-I_{k}&0_{k}\\\end{array}\right)$. Let $\tau$ be the complex conjugation on
$\SU(n)$ (and on $G_{n,m}$). When $n$ is even, let $\tau'=\tau\circ\Ad(J_{n/2})$. Any irreducible compact symmetric
space $M$ which is of type AI or AII in Cartan's notation is isomorphic to $G_{n,m}/G_{n,m}^{\tau}$ ($m|n$) or
$G_{n,m}/G_{n,m}^{\tau'}$ ($m|n$ and $n$ is even).
The isomorphism types of the groups $G_{n,m}^{\tau}$, $G_{n,m}^{\tau'}$ are as follows: \begin{enumerate}
\item[(1)] If $m$ is odd, then $G_{n,m}^{\tau}\cong\SO(n)$ and $G_{n,m}^{\tau'}\cong\Sp(n/2)$ (in case $n$ is even).
\item[(2)] If $m$ and $\frac{n}{m}$ are both even, then $G_{n,m}^{\tau}\cong\PSO(n)\times\mathbb{Z}/2\mathbb{Z}$ and
$G_{n,m}^{\tau'}\cong\PSp(n/2)\times\mathbb{Z}/2\mathbb{Z}$.
\item[(3)] If $m$ is even and $\frac{n}{m}$ is odd, then $G_{n,m}^{\tau}\cong\PO(n)$ and $G_{n,m}^{\tau'}\cong\PSp(n/2)$.
\end{enumerate}

\smallskip

\noindent\textbf{3, Types CI and DIII.} Any irreducible compact symmetric space which is of type CI or DIII in
Cartan's notation is isomorphic to $G/G^{\theta}$ for some $(G,\theta)$ as in the following list:
\begin{enumerate}
\item[(i)] adjoint type: $G=\PSp(n)$ ($n\geq 1$), $\theta=\Ad(\mathbf{i}I)$.
\item[(ii)] $G=\Sp(n)$ ($n\geq 1$), $\theta=\Ad(\mathbf{i}I)$, and $G^{\theta}=\U(n)$.
\item[(iii)] adjoint type: $G=\PSO(2n)$ ($n\geq 3$), $\theta=\Ad(J_{n})$.
\item[(iv)] $G=\SO(2n)$ ($n\geq 3$), $\theta=\Ad(J_{n})$, and $G^{\theta}=\U(n)$.
\end{enumerate}

\smallskip

\noindent\textbf{4, Irreducible compact symmetric spaces of exceptional type.} We call an irreducible compact symmetric
space $M$ of {\it exceptional type} if the neutral subgroup of its isometry group is a compact exceptional simple Lie
group. Any irreducible compact symmetric space of exceptional type is isomorphic to $G/G^{\theta}$ for some $(G,\theta)$
as in the following list:
\begin{enumerate}
\item[(i)] adjoint type: when $G$ is a connected compact simple Lie group of adjoint type, and $\theta$ is an involutive
automorphism of $G$.
\item[(ii)] $G=\E_{6}^{sc}$, $\theta\sim\sigma_3$ or $\sigma_4$ as in \cite[Table 1]{Huang-Yu}, $G^{\sigma_{3}}\cong\F_{4}$
and $G^{\sigma_{4}}\cong\PSp(4)$.
\item[(iii)] $G=\E_{7}^{sc}$, $\theta\sim\sigma_2$ or $\sigma_3$ as in \cite[Table 1]{Huang-Yu}, $G^{\sigma_{2}}\cong
(\E_{6}^{sc}\times\U(1))/\langle(c,1)\rangle$ (where $c$ is a nontrivial central element of $\E_{6}^{sc}$) and $G^{\sigma_{3}}\cong\SU(8)/\langle-I\rangle$.
\end{enumerate}

\section{Explicit classification of maximal antipodal sets}\label{S:classification}

In this section we classify maximal antipodal sets in irreducible compact symmetric spaces.

\subsection{Irreducible compact symmetric spaces of group form}\label{SS:group}

Let $M=G$ be an irreducible compact symmetric space of group form. Then, either $G\cong\U(1)$, or $G$ is a compact
simple Lie group. In this case, the geodesic symmetry $s_{x}$ acts by $s_{x}(y)=xy^{-1}x$ ($\forall x,y\in G$). Let
$X$ be a subset of $M$ containing the origin. It is clear that $X$ is a maximal antipodal set if and only if it
is a maximal elementary abelian $2$-subgroup of $G$. When $G=\U(1)$, then $X=\{\pm{1}\}$. When $G$ is of adjoint
type, maximal elementary abelian $2$-subgroups of $G$ are classified in \cite{Griess} and \cite{Yu}. The other
connected compact simple Lie groups fall into the following list:
\begin{enumerate}
\item[(i)] $\SU(n)/\langle e^{\frac{2\pi i}{m}}I\rangle$ ($m|n$, $m\neq n$).
\item[(ii)] $\Sp(n)$ ($n\geq 2$).
\item[(iii)] $\Spin(n)$ ($n\geq 7$).
\item[(iv)] $\SO(n)$ ($n\geq 8$, even).
\item[(v)] $\Spin(4m)/\langle c\rangle$ ($m\geq 2$).
\item[(vi)] $\E_{6}^{sc}$.
\item[(vii)] $\E_{7}^{sc}$.
\end{enumerate}

In item (i), let $G=\SU(n)/\langle e^{\frac{2\pi i}{m}}I\rangle$. When $m$ is odd, any maximal elementary abelian
2-subgroup is conjugate to the subgroup consisting of diagonal matrices with entries $\pm{1}$; when $m$ is even,
the map $X\mapsto\pi(X)$ with $\pi$ the projection $G\rightarrow\PSU(n)$ gives a bijection between conjugacy classes
of maximal elementary abelian 2-subgroups in $G$ and that in $\PSU(n)$. The latter is classified in
\cite[Proposition 2.4]{Yu}.

In item (ii) or item (iv), there is a unique conjugacy class of maximal elementary abelian 2-subgroups, i.e., those
conjugate to the subgroup consisting of diagonal matrices with entries $\pm{1}$.

In item (vi), due to $Z(\E_{6}^{sc})\cong\mathbb{Z}/3\mathbb{Z}$ is of odd degree, the map $X\mapsto\pi(X)$ with $\pi$
the projection $\E_{6}^{sc}\rightarrow\E_{6}^{ad}$ gives a bijection between conjugacy classes of maximal elementary
abelian 2-subgroups in $\E_{6}^{sc}$ and that in $\E_{6}^{ad}$. There are two conjugacy classes, corresponding to the
subgroups $F'_{2,3}$ and $F'_{0,1,0,2}$ in \cite[Pages 272-273]{Yu}.

In item (vii), $X\sim\pi^{-1}(X')$ with $\pi$ the projection $\E_{7}^{sc}\rightarrow\E_{7}^{ad}$, and $X'=F'''_{0,3}$
(rank 6) or $F''_{2}$ (rank 5) in \cite[Page 284]{Yu}.

We do not know yet a complete classification of maximal elementary 2-subgroups for groups in item (iii) and item (v).

\subsection{Grassmannians}\label{SS:Grassmannian-classification}

In this subsection we classify maximal antipodal sets in an irreducible compact symmetric space which is isogenous to
a Grassmannian. As stated in Section \ref{S:symmetric space}, there are six cases to consider: item (i) is the adjoint
type case, which is treated in Remark \ref{R:X-F}; for item (v) and item (vi), we do not have a full classification yet.
Below we treat items (ii)-(iv).

\begin{example}\label{E:Grass1}
Let $M=\SU(2p)/S(\U(p)\times\U(p))$ and let $X\subset M$ be a maximal antipodal set containing the origin $o$.
Write $G=\SU(2p)$. Define $\theta\in\Aut(G)$ by $$\theta(g)=I_{p,p}gI_{p,p}^{-1},\ \forall g\in G.$$ Then,
$M=G/G^{\theta}$. Set $\bar{G}=G\rtimes\langle\bar{\theta}\rangle$, where $\bar{\theta}^{2}=1$ and
$\Ad(\bar{\theta})|_{G}=\theta$.

When $p$ is odd, we may identity $\bar{G}$ with $\SU^{\pm{}}(2p)$ and identify $\bar{\theta}$ with $I_{p,p}$.
Then, $F_{2}(X)$ is diagonalizable. Without loss of generality we assume that $F_{2}(X)$ is contained in
the subgroup $F$ of $\bar{G}=\SU^{\pm{}}(2p)$ consisting of diagonal matrices with entries $\pm{1}$. Then,
$$|X|=|F_{2}(X)\cap C_{\bar{\theta}}|=|F\cap C_{\bar{\theta}}|=\binom{2p}{p}.$$

When $p$ is even, we may identify $\theta$ with $I_{p,p}$. Then, $\bar{G}=G\times\langle\bar{\theta}
\theta^{-1}\rangle$, $F_{2}(X)=F_{1}(X)\times\langle\bar{\theta}\rangle$ and $F_{2}(X)$ is diagonalizable.
Without loss of generality we assume that $F_{1}(X)$ is contained in the subgroup $F$ of $\SU(2p)$ consisting
of diagonal matrices with entries $\pm{1}$. Then, $$|X|=|F_{2}(X)\cap C_{\bar{\theta}}|=|F\cap C_{\theta}|
=\binom{2p}{p}.$$
\end{example}

\begin{example}\label{E:Grassmannian-quaternion}
When $M=\Sp(2p)/(\Sp(p)\times\Sp(p))$, the classification proceeds the same as in Example \ref{E:Grass1}: there
is a unique maximal antipodal set $X$ in $M$ up to conjugacy, and $|X|=\binom{2p}{p}$.
\end{example}

\begin{example}\label{E:Grassmannian-real}
Let $M=\SO(2p)/S(\O(p)\times\O(p))$ ($p\geq 3$), the classification proceeds the same as in Example
\ref{E:Grass1}: there is a unique maximal antipodal set $X$ in $M$ up to conjugacy, and $|X|=\binom{2p}{p}$.
\end{example}

\subsection{Types AI and AII}\label{SS:A1-2-classification}

As stated in $\S$ \ref{S:symmetric space}, any irreducible compact symmetric space $M$ which is of type AI or AII
in Cartan's notation is isomorphic to $G/G^{\theta}$ where $G=G_{n,m}$ and $\theta=\tau$ or $\tau'$. When $m=n$,
$M$ is of adjoint type and it is treated in Remark \ref{R:X-F}. According to \cite[Propositions 2.12 and 2.16]{Yu},
there are $k+1$ conjugacy classes of maximal elementary abelian 2-subgroups in $\PO(n)=G^{\tau}$ (or $\PSp(n/2)=
G^{\tau'}$), where $k$ is the 2-power index of $n$ (or $\frac{n}{2}$).

When $m=1$, we have $M\cong\SU(n)/\SO(n)$ or $\SU(n)/\Sp(n/2)$.

\begin{example}\label{E:A1}
Let $M=\SU(n)/\SO(n)$ and let $X\subset M$ be a maximal antipodal set containing the origin $o$. Write $G=\SU(n)$
and $\theta=\tau\in\Aut(G)$. Then, $M=G/G^{\theta}$. Set $\bar{G}=G\rtimes\langle\bar{\theta}\rangle$, where
$\bar{\theta}^{2}=1$ and $\Ad(\bar{\theta})|_{G}=\theta$. Taking similar study as in Example \ref{E:Grass1}, we
have $F_{1}(X)\subset G^{\theta}=\SO(n)$. Then, $F_{1}(X)$ is conjugate to the subgroup of $\SO(n)$ consisting
of diagonal matrices with entries $\pm{1}$ and $$|X|=|F_{2}(X)\cap C_{\bar{\theta}}|=|F_{1}(X)|=2^{n-1}.$$
\end{example}

\begin{example}\label{E:A2}
When $M=\SU(n)/\Sp(\frac{n}{2})$ ($n\geq 4$, even), the classification is along the same line as in Example
\ref{E:A1} by replacing $\tau$, $\SO(n)$ with $\tau'$, $\Sp(\frac{n}{2})$ respectively. The result is: there is
a unique maximal antipodal set $X$ in $M$ up to conjugacy and $|X|=2^{\frac{n}{2}-1}$.
\end{example}

In general, when $m$ is odd, the classification is the same as in the case of $m=1$. When $n/m$ is odd, the
classification is the same as in the adjoint type case. When $m$ and $n/m$ are both even, we have $G^{\tau}\cong
\PSO(n)\times\mathbb{Z}/2\mathbb{Z}$ and $G^{\tau'}\cong\PSp(n/2)\times\mathbb{Z}/2\mathbb{Z}$. Using the
classification of elementary abelian 2-subgroups of $\PSO(n)$ and of $\PSp(\frac{n}{2})$ given in \cite{Yu},
one can classify maximal antipodal sets.

\subsection{Types CI and DIII}\label{SS:C1-D3-classification}

Let $M$ be a compact symmetric space of type CI or DIII. Item (i) as listed in $\S$ \ref{S:symmetric space}
is treated in Remark \ref{R:X-F}. We treat items (ii)-(iii) below.

\begin{example}\label{E:D3}
Let $M=\SO(2n)/\U(n)$ and let $X\subset M$ be a maximal antipodal set containing the origin $o$. Write $G=\SO(2n)$
and $\theta=\Ad(J_{n})\in\Aut(G)$. Then, $M=G/G^{\theta}$. Set $\bar{G}=G\rtimes\langle\bar{\theta}\rangle$,
where $\bar{\theta}^{2}=1$ and $\Ad(\bar{\theta})|_{G}=\theta$. Taking similar study as in Example \ref{E:Grass1},
we have $F_{1}(X)\subset G^{\theta}=\U(n)$. Then, $F_{1}(X)$ is conjugate to the subgroup of $\U(n)$ consisting
of diagonal matrices with entries $\pm{1}$ and with determinant $1$ (this condition is forced by $F_{2}(X)$ is
generated by elements in $C_{\bar{\theta}}$) and $$|X|=|F_{2}(X)\cap C_{\bar{\theta}}|=|F_{1}(X)|=2^{n-1}.$$
\end{example}

\begin{example}\label{E:C1}
Let $M=\Sp(n)/\U(n)$ and let $X\subset M$ be a maximal antipodal set containing the origin $o$. The classification
is similar to Example \ref{E:D3}: there is a unique maximal antipodal set $X$ in $M$ up to conjugacy and $|X|=2^{n}$.
\end{example}

\subsection{Exceptional type}\label{SS:exceptional-classification}

Let $M$ be an irreducible compact symmetric space of exceptional type. Item (i) as listed in $\S$ \ref{S:symmetric space}
is treated in Remark \ref{R:X-F}. We treat items (ii)-(iii) below.

\begin{example}\label{E:E6}
Let $M=\E_{6}^{sc}/(\E_{6}^{sc})^{\theta}$ for $\theta$ an outer involution, and let $X\subset M$ be a maximal
antipodal set containing the origin $o$. Write $G=\E_{6}^{sc}$. Set $\bar{G}=G\rtimes\langle\bar{\theta}\rangle$,
where $\bar{\theta}^{2}=1$ and $\Ad(\bar{\theta})|_{G}=\theta$. Let $\pi:\bar{G}\rightarrow\Aut(\mathfrak{e}_{6})$
be the adjoint homomorphism.

When $\theta\sim\sigma_{3}$, by \cite[Proposition 6.3]{Yu} one shows that $\pi(F_{2}(X))$ is conjugate to the
subgroup $F_{2,0}$ of $\F_4=G^{\theta}$. Then, $|X|=4$.

When $\theta\sim\sigma_{4}$ and $\pi(F_{2}(X))$ contains no element conjugate to $\sigma_{3}$, by
\cite[Proposition 6.5]{Yu} one shows that $\pi(F_{2}(X))$ is conjugate to the subgroup $F_{0,1,0,2}$ of
$\Aut(\fre_{6})$. Then, $|X|=64$. When $\theta\sim\sigma_{4}$ and $\pi(F_{2}(X))$ contains an element conjugate
to $\sigma_{3}$, by \cite[Proposition 6.3]{Yu} one shows that $\pi(F_{2}(X))$ is conjugate to the subgroup
$F_{2,3}$ of $\Aut(\fre_{6})$. Then, $|X|=2^{5}-2^{2}=28$.
\end{example}

\begin{example}\label{E:E7}
Let $M=\E_{7}^{sc}/(\E_{7}^{sc})^{\theta}$ for $\theta\sim\sigma_2$ or $\theta\sim\sigma_3$, and let $X\subset M$
be a maximal antipodal set containing the origin $o$. Write $G=\E_{7}^{sc}$. Set $\bar{G}=G\rtimes\langle
\bar{\theta}\rangle$, where $\bar{\theta}^{2}=1$ and $\Ad(\bar{\theta})|_{G}=\theta$. Let $\pi:\bar{G}\rightarrow
\Aut(\mathfrak{e}_{7})$ be the adjoint homomorphism.

When $\theta\sim\sigma_{2}$, taking similar study as in Example \ref{E:Grass1} we have \[F_{1}(X)\subset G^{\theta}
\cong(\E_{6}\times\U(1))/\langle(c,e^{\frac{2\pi i}{3}})\rangle.\] Write $c=[(1,-1)]\in(\E_{6}\times\U(1))/\langle
(c,e^{\frac{2\pi i}{3}})\rangle$. As shown in \cite[\S 7.1]{Yu}, for any element $x\in F_{1}(X)$, $x\bar{\theta}
\in C_{\bar{\theta}}$ if and only if $x\sim 1$, $c$, $\tau_{2}$ or $c\tau_{2}$ in $(\E_{6}\times\U(1))/\langle
(c,e^{\frac{2\pi i}{3}})\rangle$. Then, $F_{1}(X)$ is of the form $F_{1}(X)=J\times\langle(1,-1)\rangle$, where $J$
is an elementary abelian 2-subgroup of $\E_6$. By \cite[Proposition 6.5]{Yu} one can show that $\pi(F_{2}(X))$ is
conjugate to $F'_{0,1,0,2}\subset\E_{7}^{ad}$. Counting conjugacy classes of elements in $F'_{0,1,0,2}$ we get
$|X|=2\times\frac{2^{6}-2^{3}}{2}=56$.

When $\theta\sim\sigma_{3}$, taking similar study as in Example \ref{E:Grass1} we have \[F_{1}(X)\subset G^{\theta}
\cong\SU(8)/\langle -I\rangle.\] For any element $x\in F_{1}(X)$, $x\bar{\theta}\in C_{\bar{\theta}}$
if and only if $x$ is conjugate to $[I]$, $[\mathbf{i}I]$, $[I_{4,4}]$ or $[\mathbf{i}I_{4,4}]$ in
$\SU(8)/\langle -I\rangle$. Choose a maximal elementary abelian 2-subgroup $F$ of $\PSU(8)$ containing the image
of projection $F$ of $F_{1}(X)$ in it. As in \cite[\S 2]{Yu}, it is associated with a multiplicative function
$m:F\times F\rightarrow\{\pm{1}\}$. Put $r=\rank\ker m+1$ and $s=\frac{1}{2}\rank(F/\ker m)$. Then, $r\cdot 2^{s}=8$.
Then, $(r,s)=(8,0)$, $(4,1)$ or $(1,3)$. When $(r,s)=(8,0)$, we have $|X|=72$; when $(r,s)=(4,1)$, we have
$|X|=2^{4}\cdot 3+2^{3}=56$; when $(r,s)=(1,3)$, we have $|X|=128$.
\end{example}

\section{Supplements}

\subsection{Characterization of polars}\label{S:anitipodal2}

Let $M$ be a compact symmetric space. Connected components of the fixed point set of the geodesic symmetry $s_{x}$
at a point $x\in M$ are called polars by Chen and Nagano and are classified in \cite{Chen-Nagano2} and \cite{Nagano}.
Now let $M=G/G^{\theta}$ for $G$ a connected compact simple Lie group and $\theta$ an involutive automorphism
of $G$. We remark here that results in \cite{Huang-Yu} also apply to classify $G^{\theta}$ orbits in the fixed
point set of $s_{o}$. When $G^{\theta}$ is connected, this is equivalent to the classification of polars.
In general, $\pi_{0}(G^{\theta})=(\mathbb{Z}/2\mathbb{Z})^{r}$ ($r=0,1,2$) and $r=2$ happens only when
$M=\PSO(8)/\PSO(8)^{[\Ad(I_{4,4})]}$ (\cite[Table 2, p. 408]{Huang-Yu}). Thus, a $G^{\theta}$ orbit is the union
of 1,2 or 4 polars. Set $\bar{G}=G\rtimes\langle\bar{\theta}\rangle$ where $\bar{\theta}^{2}=1$ and
$\Ad(\bar{\theta})|_{G}=\theta.$ Write $C_{\bar{\theta}}=\{g\bar{\theta}g^{-1}:g\in G\}.$ The classification is
based on the following lemma, which is easy and we omit the proof.

\begin{lemma}\label{L:characterization2'}
A point $x=gG^{\theta}\in G/G^{\theta}$ is in the fixed point set of $s_{o}$ if and only if $\phi(x)\in G^{\theta}$
and $\phi(x)^{2}=1$. The $G^{\theta}$ orbits in the fixed point set of $s_{o}$ are in one-to-one correspondence with
$G$ orbits of ordered pairs $(\theta_1,\theta_2)\in\bar{G}\times\bar{G}$ such that $\theta_1,\theta_2\in
C_{\bar{\theta}}$ and $\theta_1\theta_2=\theta_2\theta_1$.
\end{lemma}

When $G$ is of adjoint type, ordered pairs of commuting involutions in $\bar{G}$ are classified in \cite{Huang-Yu}.
When $G$ is not of adjoint type, the classification can be made by considering the projection $\pi: G\rightarrow
\Int(\fru_0)$ and using the classification in \cite{Huang-Yu}. For any $o\neq x=gG^{\theta}\in G/G^{\theta}$,
put $\theta_1=\bar{\theta}$ and $\theta_2=g\bar{\theta}g^{-1}$. Then, \[\Stab_{G^{\theta}}(x)=G^{\theta}\cap gG^{\theta}g^{-1}=Z_{G}(\langle\theta_{1},\theta_{2}\rangle).\] The group $\langle\theta_{1},\theta_{2}\rangle$
is a Klein four subgroup of $G$, the centralizers $Z_{G}(\langle\theta_{1},\theta_{2}\rangle)$ are calculated in
\cite[Table 6, p. 420]{Huang-Yu} when $G$ is of adjoint type. When $G$ is not of adjoint type, one can apply the
method in \cite{Huang-Yu} to calculate the centralizers $Z_{G}(\langle\theta_{1},\theta_{2}\rangle)$ as well.

\subsection{Some corrections to \cite{Yu}}

Here I would like to make several corrections to \cite{Yu}.
In \cite[p. 273, lines 7-9]{Yu}, the correct definition for the groups $F_{\epsilon,\delta,r,s}$ ($\epsilon+\delta\leq 1$,
$r+s\leq 2$) should be\[F_{\epsilon,\delta,r,s}=\left\{\begin{array}{cc}\langle x_{0},x_{1},\dots,x_{\epsilon+2\delta},
x_{3},\dots,x_{2+r+2s}\rangle\textrm{ if }(r,s)\neq(2,0)\\ \langle x_{0},x_{1},\dots,x_{\epsilon+2\delta},x_{3},x_{5}
\rangle\textrm{ if }(r,s)=(2,0).\\\end{array}\right.\] Accordingly, $F'_{\epsilon,\delta,r,s}$ ($\epsilon+\delta\leq 1$,
$r+s\leq 2$) should be defined by \[F'_{\epsilon,\delta,r,s}=\left\{\begin{array}{cc}\langle x_{1},\dots,x_{\epsilon+
2\delta},x_{3},\dots,x_{2+r+2s}\rangle\textrm{ if }(r,s)\neq(2,0)\\ \langle x_{1},\dots,x_{\epsilon+2\delta},x_{3},
x_{5}\rangle\textrm{ if }(r,s)=(2,0).\\\end{array}\right.\] I would like to thank Alastair Litterick, Heiko Dietrich,
Haian He for pointing out these two mistakes.

In \cite[p. 291, lines -1]{Yu}, it should be $C=\Gamma_{3}$. In \cite[p. 291, lines -12 - -11]{Yu}, it should be
\[F''_{r,0}=A^{r}\times B,\]  \[F''_{r,1}=A^{r}\times C,\] \[F''_{r,2}=A^{r}\times D,\] where $B,C,D$ are elementary
abelian 2-subgroups with rank equal to 1,2,3 respectively and each has a unique element conjugate to $\sigma_1$.

These mistakes do not affect the statement of any result in \cite{Yu}.

\subsection{Open cases}\label{SS:open}

In summary, the only irreducible compact symmetric spaces for which we do not have a complete classification
of maximal antipodal sets yet are in the following list:
\begin{enumerate}
\item[(i)] $M=\Spin(n)$ ($n\geq 7$).
\item[(ii)] $M=\Spin(4n)/\langle c\rangle$ ($n\geq 3$).
\item[(iii)] $M=\Spin(p+q)/\Spin(p)\cdot\Spin(q)$ ($p\geq q\geq 1$ and $p+q\geq 7$).
\item[(iv)] $M=G/G^{\theta}$ where $G=\Spin(4n)/\langle c\rangle$ and $\theta=\Ad(e_{1}e_{2}\dots e_{2n})$.
\end{enumerate}

For any $k\geq 1$, identify the $\mathbb{Z}/2\mathbb{Z}$-vector space $V_{k}=(\mathbb{Z}/2\mathbb{Z})^{k}$ with
the set of subsets of $\{1,\dots,k\}$ and denote by $e_{I}\in V_{k}$ for an element corresponding to a subset $I$
of $\{1,\dots,k\}$. Define an anti-symmetric form on $V_{k}$ by $(e_{I},e_{J})=|I\cap J|\pmod{2}$. Let $V'_{k}$
be the subspace of $e_{I}\in V_{k}$ such that $\sharp{I}$ is even. A sub-vector space $W$ of $V'_{k}$ is said to
be an isotropic subspace if $(e_{I},e_{J})=0$ for any $e_{I},e_{J}\in W$; an isotropic subspace $W$ of $V'_{k}$
is called a Lagrangian if it is not properly contained in any other isotropic subspace. Write $X_{k}$ for the
set of Lagrangians in $V'_{k}$ and write $X'_{k}$ for the subset of $X_{k}$ consisting of Lagrangians
$W\subset V'_{k}$ such that $|I|\neq 2$ for any $e_{I}\in W$. Then, both $X_{k}$ and $X'_{k}$ admit actions of
the permutation group $S_{k}$. Write $X_{k}/S_{k},X'_{k}/S_{k}$ for the corresponding orbit sets. For any
$W\in X_{k}$ (or $W\in X'_{k}$), write $[W]\in X_{k}/S_{k}$ (or $[W]\in X'_{k}/S_{k}$) for the $S_{k}$ orbit
containing $W$.

The following proposition says something for maximal antipodal sets in $\Spin(n)$.

\begin{prop}
Let $n\geq 1$. Then: \begin{enumerate}
\item[(1)]the cardinality of each maximal antipodal set in $\Spin(n)$ is equal to $2^{\lfloor\frac{n+2}{2}\rfloor}$;
\item[(2)]the orbit set of maximal antipodal sets in $\Spin(n)$ can be parametrized by the set $X_{n}/S_{n}$;
\item[(3)]there is a decomposition \[X_{n}/S_{n}\cong\bigsqcup_{0\leq r\leq\lfloor\frac{n}{2}\rfloor}X'_{n-2r}/S_{n-2r}.\]
\item[(4)]$X'_{k}=\emptyset$ if $k\in\{2,3,4,5,6\}$.
\end{enumerate}
\end{prop}

\begin{proof}[Sketch of proof]
We show a correspondence between maximal antipodal sets in $\Spin(n)$ and Lagrangians in $V'_{n}$.
Let $F$ be a maximal antipodal set in $\Spin(n)$. Without loss of generality we assume that $1\in F$. Then,
$F$ is a maximal elementary abelian 2-subgroup of $\Spin(n)$. Thus, $Z(\Spin(n))\subset F$. Consider the natural
projection $\pi:\Spin(n)\rightarrow\SO(n)$. Since any elementary abelian 2-subgroup of $\SO(n)$ is conjugate to
a diagonal one, we assume that $\pi(F)$ is contained in the subgroup $F'_0$ of diagonal matrices in $\SO(n)$.
Identify $F'_0$ with the $\mathbb{Z}/2\mathbb{Z}$-vector space $V'_{n}\subset V_{n}=(\mathbb{Z}/2\mathbb{Z})^{n}$,
and also the set of subsets $I$ of $\{1,\dots,n\}$ with $\sharp{I}$ even. Let $W\subset V'_{n}$ correspond to
$\pi(F)$. We have $[e_{I},e_{J}]=(-1)^{I\cap J}\in\Spin(n)$ (the repetition of the notation $e_{I}$ to mean
either an element in $V_{I}$ or an element in $\Spin(n)$ is cute, here $e_{I},e_{J}$ means elements in $\Spin(n)$)
for any two subsets $I,J$ of $\{1,\dots,n\}$ with $e_{I},e_{J}\in W$. Then, $F$ is a maximal elementary abelian
2-subgroup if and only if $W$ is a Langrangian. This shows the assertion (1).

The assertion (2) can be shown in an inductive way using two facts: (a)$\pi(e_{I})$ and $\pi(e_{J})$ are conjugate
in $\SO(n)$ if and only if $I$ and $J$ and in the same $S_{n}$ orbit; (b)for any set of elements $e_{I_{1}},\dots,
e_{I_{s}}$ of $F$, the centralizer of $\langle\pi(e_{I_{1}}),\dots,\pi(e_{I_{s}})\rangle$ in $\O(n)$ is
a product of $\O(n_{j})$ ($1\leq j\leq t$) where $\sum_{1\leq j\leq t}n_{j}=n$.

The assertion (3) is easy to show. The assertion (4) can be shown by a case by case verification.
\end{proof}

For items (ii)-(iv), I even don't know cardinalities of maximal antipodal sets except when $n$ or $\min\{p,q\}$
is small.


\begin{thebibliography}{999}
%\bibitem {Berger} M. Berger, \emph{Les espaces sym\'etriques noncompacts}, Ann. Sci. \'Ecole Norm. Sup.
%(3) \textbf{74} (1957) 85--177.

%\bibitem {Beyrer} J.~Beyrer, \emph{A complete description of the antipodal set of most symmetric spaces of compact type.}
%Osaka J. Math. \textbf{55} (2018), no. 3, 567-586.

%\bibitem {Bourbaki} N.~Bourbaki, \emph{Lie groups and Lie algebras.} Chapters 4�C6. Translated from the 1968
%French original by Andrew Pressley. Elements of Mathematics (Berlin). Springer-Verlag, Berlin, 2002.

\bibitem{Borel} A.~Borel, \emph{Semisimple groups and Riemannian symmetric spaces.} Texts and Readings in Mathematics,
\textbf{16}. Hindustan Book Agency, New Delhi, (1998).

\bibitem {Chen} B.Y.~Chen, \emph{Two-numbers and their applications-a survey.} Bull. Belg. Math. Soc. Simon Stevin \textbf{25}
(2018), no. 4, 565-596.

\bibitem {Chen-Nagano2} B.Y.~Chen; T.~Nagano, \emph{Totally geodesic submanifolds of symmetric spaces. II.} Duke Math. J.
\textbf{45} (1978), no. 2, 405-425.

\bibitem {Chen-Nagano} B.Y.~Chen; T.~Nagano, \emph{A Riemannian geometric invariant and its applications to a
problem of Borel and Serre.} Trans. Amer. Math. Soc. \textbf{308} (1988), no. 1, 273-297.

%\bibitem {Dynkin} Dynkin E.-B., \emph{Semisimple subalgebras of semisimple Lie algebras}, AMS translation,
%(2) \textbf{6} (1957) 111-254.

\bibitem {Griess} R.L.~Griess, Jr., \emph{Elementary abelian $p$-subgroups of algebraic groups}.
Geom. Dedicata \textbf{39} (1991),  no. 3, 253-305.

\bibitem {Helgason} S.~Helgason, \emph{Differential geometry, Lie groups, and symmetric spaces}. Corrected
reprint of the 1978 original. Graduate Studies in Mathematics, \textbf{34}. American Mathematical Society,
Providence, RI, (2001).

\bibitem {Huang-Yu} J.-S.~Huang; J.~Yu, \emph{Klein four subgroups of Lie algebra automorphisms.}
Pacific J. Math. \textbf{262} (2013), no. 2, 397-420.

%\bibitem {Knapp} A. W. Knapp, \emph{Lie groups beyond an introduction}. Second edition. Progress in Mathematics,
%140. Birkh\"auser Boston, Inc., Boston, MA, 2002.

%\bibitem {Liu-Deng} X.D.~Liu; S.Q.~Deng, \emph{The antipodal sets of compact symmetric spaces.} Balkan J. Geom.
%Appl. \textbf{19} (2014), no. 1, 73-79.

\bibitem{Loos1} O.~Loos, \emph{Symmetric spaces. I: General theory.} W. A. Benjamin, Inc., New York-Amsterdam, (1969).

\bibitem{Loos2} O.~Loos, \emph{Symmetric spaces. II: Compact spaces and classification.} W. A. Benjamin, Inc.,
New York-Amsterdam, (1969).

\bibitem{Nagano} T.~Nagano, \emph{The involutions of compact symmetric spaces.} Tokyo J. Math. \textbf{11} (1988),
no. 1, 57-79.

\bibitem{Takeuchi} M.~Takeuchi, \emph{On the fundamental group and the group of isometries of a symmetric space.}
J. Fac. Sci. Univ. Tokyo Sect. I \textbf{10} (1964), 88–123.

\bibitem {Tanaka} M.~Tanaka, \emph{Antipodal sets of symmetric R-spaces.} Osaka J. Math. \textbf{50}
(2013), no. 1, 161-169.

\bibitem {Tanaka-Tasaki1} M.~Tanaka; H.~Tasaki, \emph{Antipodal sets of symmetric R-spaces.} Osaka J. Math. \textbf{50}
(2013), no. 1, 161-169.

\bibitem {Tanaka-Tasaki2} M.~Tanaka; H.~Tasaki, \emph{Maximal antipodal subgroups of some compact classical Lie groups.}
J. Lie Theory \textbf{27} (2017), no. 3, 801-829.

\bibitem {Tanaka-Tasaki3} M.~Tanaka; H.~Tasaki, \emph{Maximal antipodal sets of compact classical symmetric spaces and
their cardinalities I.}  Differential Geom. Appl. \textbf{73} (2020), 101682, 32 pp.

%\bibitem {Tirao} J.A.~Tirao, \emph{Antipodal manifolds in compact symmetric spaces of rank one.} Proc. Amer. Math. Soc.
%\textbf{72} (1978), no. 1, 143-149.

\bibitem{Wolf} J.A.~Wolf, \emph{Spaces of constant curvature.} Fifth edition. Publish or Perish, Inc., Houston, TX, (1984).

\bibitem {Yu} J.~Yu, \emph{Elementary abelian 2-subgroups of compact Lie groups}. Geom. Dedicata
\textbf{167} (2013), no. 1, 245-293.

%\bibitem {Yu2} Yu, J., \emph{A note on closed subgroups of compact Lie groups}. arXiv:0912.4497v1

\end{thebibliography}
\end{document}